\documentclass{amsart}
\usepackage{amsmath,amssymb,amsthm}
\newtheorem{thr}{Theorem}

\newtheorem{con}[thr]{Conjecture}

\theoremstyle{definition}

\theoremstyle{remark}

\newtheorem{obs}[thr]{Observation}

\numberwithin{equation}{section}

\begin{document}

\title{Counterexamples on spectra of sign patterns}

\author{Yaroslav Shitov}
\address{National Research University Higher School of Economics, 20 Myasnitskaya Ulitsa, Moscow 101000, Russia}
\email{yaroslav-shitov@yandex.ru}

%    \subjclass is required.

\begin{abstract}
An $n\times n$ \textit{sign pattern} $S$, which is a matrix with entries $0,+,-$, is called \textit{spectrally arbitrary} if any monic real polynomial of degree $n$ can be realized as a characteristic polynomial of a matrix obtained by replacing the non-zero elements of $S$ by numbers of the corresponding signs. A sign pattern $S$ is said to be a \textit{superpattern} of those matrices that can be obtained from $S$ by replacing some of the non-zero entries by zeros. We develop a new technique that allows us to prove spectral arbitrariness of sign patterns for which the previously known \textit{Nilpotent Jacobian} method does not work. Our approach leads us to solutions of numerous open problems known in the literature. In particular, we provide an example of a sign pattern $S$ and its superpattern $S'$ such that $S$ is spectrally arbitrary but $S'$ is not, disproving a conjecture proposed in 2000 by Drew, Johnson, Olesky, and van den Driessche.
\end{abstract}

\maketitle

\section{Conjectures}

The study of spectra of matrix patterns deserved a significant amount of attention in recent publications. %, including the papers~\cite{BMOD, CV, DJOD} and numerous works that cite them.
The conjecture mentioned in the abstract appeared in one of the foundational papers on this topic~(\cite{DJOD}), and many subsequent works proved it in different special cases~(\cite{CF, CM, GS2, KSVW, MTD}). One of the known sufficient conditions for superpatterns to be spectrally arbitrary is the \textit{Nilpotent Jacobian} condition~(\cite{BMOD, DJOD}), which allowed to solve several intriguing problems on this topic~(\cite{CKSV, GS, Per}). Despite these efforts, the \textit{superpattern conjecture} remained open by now, and we mention~\cite{COD, McY, Mel} as other recent work disussing this conjecture.

\begin{con}\label{conj11}\emph{(Conjecture 16 in~\cite{DJOD}.)}
If $S$ is a minimal spectrally arbitrary sign pattern, then any superpattern of $S$ is spectrally arbitrary.
\end{con}

We note that this conjecture involves the concept of a \textit{minimal} spectrally arbitrary sign pattern, that is, a sign  pattern $S$ which is spectrally arbitrary but is not a superpattern of any other spectrally arbitrary sign pattern. In our paper, we construct a sign pattern $S$ and its superpattern $S'$ such that $S$ is spectrally arbitrary but $S'$ is not. We do not investigate the question of minimality of $S$, but $S$ is anyway a superpattern of some minimal spectrally arbitrary pattern $S_0$, and the pair $(S_0,S')$ provides a counterexample to Conjecture~\ref{conj11} even if $S$ is not minimal. %, and we construct these $S$ and $S'$ in the rest of our paper.

%Our approach allows us to solve several other problems known in the literature.

As said above, the \textit{Nilpotent Jacobian} condition is sufficient for a zero pattern (and every superpattern of it) to be spectrally arbitrary. Our results show that this condition is not necessary, answering the questions posed explicitly in~\cite{BVT, DJOD, McY}.

As a byproduct of our approach, we obtain solutions of two other related problems on the topic. Namely, we construct a sign pattern $U$ such that $\operatorname{diag}(U,U)$ is spectrally arbitrary but $U$ itself is not. This gives a solution to the problem posed in Section~5 in~\cite{DOD} and an answer to Question~3 in~\cite{COD}.

An $n\times n$ sign pattern $S$ is said to allow \textit{arbitrary refined inertias} if, for any family $n_+,n_-,n_0,n_i$ of nonnegative integers such that $n_++n_-+n_0+2n_i=n$, there is a matrix with sign pattern $S$ which has $n_+$ eigenvalues with positive real part, $n_-$ eigenvalues with negative real part, $n_0$ zero eigenvalues, and $n_i$ purely imaginary eigenvalues. We provide an example of a sign pattern that allows arbitrary refined inertias but is not spectrally arbitrary, which solves the problem asked in Section~5 in~\cite{DOD} and in Section~5 in~\cite{KOSD}.

\section{Counterexamples}

We define the sign patterns
$$T=\begin{pmatrix}
+&+&0&0&0&0\\
-&-&+&0&0&0\\
0&0&0&+&0&0\\
0&0&0&0&+&0\\
-&-&0&0&0&+\\
+&+&+&0&-&0
\end{pmatrix},\,\,\,
T'=\begin{pmatrix}
+&+&0&0&0&0\\
-&-&+&0&0&0\\
+&0&0&+&0&0\\
0&0&0&0&+&0\\
-&-&0&0&0&+\\
+&+&+&0&-&0
\end{pmatrix},$$
which agree at every entry except $(3,1)$, so $T'$ is indeed a superpattern of $T$. Also, we fix any $2\times 2$ spectrally arbitrary pattern\footnote{In fact, spectrally arbitrary $n\times n$ sign patterns exist for all $n\geqslant 2$, see~\cite{MOTD}.} and denote it by $D$. Let us prove several observations, which we will put together in the theorem below. %We define $S=\operatorname{diag}(T,D)$, $S'=\operatorname{diag}(T',D)$, and it is clear that $S'$ is a superpattern of $S$. First, we prove that $S'$ is not spectrally arbitrary.

\begin{obs}\label{obs1}
Let $R$ be a matrix obtained from $T'$ by replacing the signs with non-zero real numbers. Then $R$ is not nilpotent.
\end{obs}

\begin{proof}
%Since $R$ is block-diagonal, it cannot be nilpotent unless its $6\times 6$ submatrix $R$ corresponding to $T'$ is nlipotent. Computing the characteristic polynomial of $R$, we get that
The coefficients of $t^3$ and $t^5$ in the characteristic polynomial of $R$ are equal to $-r_{12} r_{23} r_{31} + (r_{11} + r_{22}) r_{56} r_{65}$ and $-r_{11}-r_{22}$, respectively. These coefficients can vanish simultaneously only if $r_{12} r_{23} r_{31}=0$.
\end{proof}

\begin{obs}\label{obs11}
Let $R$ be a matrix obtained from $T$ by replacing the signs with non-zero real numbers. Assume that $t^6+a_5t^5+a_4t^4+a_3t^3+a_2t^2+a_1t+a_0$ is the characteristic polynomial of $R$. Then $a_3=0$ if and only if $a_5=0$. 
%a_3=0$, $a_5\neq0$, then the
%If $a_3=0$, $a_5\neq0$, then $t^6+a_5t^5+a_5t^5+a_4t^4+a_3t^3+a_2t^2+a_1t+a_0$ cannot be realized as a characteristic polynomial of a matrix with sign pattern $T$.
\end{obs}

\begin{proof}
The coefficients of $t^3$ and $t^5$ in the characteristic polynomial of $R$ are equal to $(r_{11} + r_{22}) r_{56} r_{65}$ and $-r_{11}-r_{22}$, respectively. Therefore, the former of them is zero if and only if the latter is zero.
\end{proof}

\begin{obs}\label{obs12}
The sign pattern $\operatorname{diag}(T,D)$ is not spectrally arbitrary.
\end{obs}

\begin{proof}
If $f=(t^2+t+1) (t^2-t+2) (t^2+1) (t^2-1)$ is realizable as the characteristic polynomial of a matrix with sign pattern $\operatorname{diag}(T,D)$, then $f$ has a divisor realizable as the characteristic polynomial of a matrix with sign pattern $T$. A straightforward checking of possible cases leads to a contradiction with Observation~\ref{obs11}.
\end{proof}

%In order to prove the spectral arbitrariness of $S$, we consider the matrix
In order to proceed, we consider the matrix
$$X=\begin{pmatrix}
x_1&1&0&0&0&0\\
-x_4&-x_2&1&0&0&0\\
0&0&0&1&0&0\\
0&0&0&0&1&0\\
-x_6&-x_5&0&0&0&1\\
x_7&x_8&x_9&0&-x_3&0
\end{pmatrix},$$
whose sign pattern is $T$ whenever the $x_i$'s take positive values.

\begin{obs}\label{obs2}
For all $b,c,d$, there are positive values of the $x_i$'s such that the characteristic polynomial of $X$ equals $(t^2+b)(t^2+c)(t^2+d)$.
\end{obs}

\begin{proof}
First, we assume that $x_1,x_3,x_8,x_9$ are arbitrary and check that the matrix $X$ defined by
$x_2 = x_1$, $x_4 = b + c + d + x_1^2 - x_3$,
$x_5 =  b c + b d + c d - b x_3 - c x_3 - d x_3 + x_3^2 + x_9$,
$x_6 = b c x_1 + b d x_1 + c d x_1 - b x_1 x_3 - c x_1 x_3 - d x_1 x_3 + x_1 x_3^2 + x_8 + x_1 x_9$,
$x_7 = -b c d + x_1 x_8 - b x_9 - c x_9 - d x_9 + x_3 x_9$
has a desired characteristic polynomial. Picking $x_3=1$ and defining $x_1$ as a large enough positive number, we make $x_2, x_3, x_4$ positive regardless of the values of $x_8, x_9$. Finally, the choice of $x_9$ allows us to make $x_5$ positive, and now $x_6,x_7$ tend to $+\infty$ as $x_8$ gets large. %Therefore, we can assign positive values to all the $x_i$'s so that the characteristic polynomial of $X$ is as desired. 
\end{proof}

\begin{obs}\label{obs3}
If $a_3/a_5>0$, then there are positive values of the $x_i$'s such that the characteristic polynomial of $X$ equals $t^6+a_5t^5+a_4t^4+a_3t^3+a_2t^2+a_1t+a_0$.
\end{obs}

\begin{proof}
Again, we assume that $x_1,x_8,x_9$ are arbitrary and check that the matrix $X$ defined by
$x_2 = a_5 + x_1$,
$x_3 = a_3/a_5$,
$x_4 = (-a_3 + a_4 a_5 + a_5^2 x_1 + a_5 x_1^2)/a_5$,
$x_5 = (a_3^2 - a_3 a_4 a_5 + a_2 a_5^2 + a_5^2 x_9)/a_5^2$,
$x_6 = (a_1 a_5^2 + a_3^2 x_1 - a_3 a_4 a_5 x_1 + a_2 a_5^2 x_1 + a_5^2 x_8 + a_5^3 x_9 + a_5^2 x_1 x_9)/a_5^2,$
$x_7 = (-a_0 a_5 + a_5 x_1 x_8 + a_3 x_9 - a_4 a_5 x_9)/a_5$
has a desired characteristic polynomial. Defining $x_1$ as a large enough positive number, we make $x_2, x_3, x_4$ positive regardless of the values of $x_8, x_9$. As in the proof of the previous observation, the choice of $x_9$ allows us to make $x_5$ positive, and then $x_6,x_7$ tend to $+\infty$ as $x_8$ gets large.
\end{proof}

\begin{obs}\label{obs4}
Let $f$ be a monic real polynomial of degree $16$. Then $f$ has a divisor realizable as the characteristic polynomial of a matrix with sign pattern $T$.
%either as in Observation~\ref{obs2} or as in Observation~\ref{obs3}.
\end{obs}

\begin{proof}
Clearly, $f$ is the product of eight quadratics of the form $t^2+a_it+b_i$. If $b_i$ is negative, then such a quadratic has two roots of different signs, and this allows us to assume that at least seven of the initial quadratics have their $b_i$'s nonnegative. By the pigeonhole principle, among these seven quadratics there are three that either have all $a_i$'s positive, or all $a_i$'s negative, or all $a_i$'s zero. In the first two cases, the product of these three quadratics is a polynomial as in Observation~\ref{obs3}, and the case of zero $a_i$'s corresponds to Observation~\ref{obs2}.
%
%If three of these seven quadratics have zero $a_i$'s, then their product is a polynomial as in Observation~\ref{obs2}. So we can assume that five of the initial quadratics have positive $b_i$'s and nonzero $a_i$'s. 
%
%
%not all of the $a_i$'s are zero, then we pick three of these quadratics such that their $a_i$'s have a non-zero sum, and the product of these three quadratics is a polynomial as in Observation~\ref{obs3}.
%
%If the $a_i$'s are zero and at least three $b_i$'s are nonnegative, then the polynomial as in Observation~\ref{obs2} divides $f$. Therefore, we can assume that $f$ has at least four non-zero roots, but then we can express $f$ as in the first paragraph but with non-zero $a_i$'s.
\end{proof}

\begin{obs}\label{obs5}
Let $V=\operatorname{diag}(T,\ldots,T,D,\ldots,D)$ be a sign pattern of size $(6t+2d)$. ($T$ occurs $t$ times, $D$ occurs $d$ times.) If $d\geqslant5$, then $V$ is spectrally arbitrary.
\end{obs}

\begin{proof}
The result is true for $t=0$ because $D$ is spectrally arbitrary (see also Proposition~2.1 in~\cite{alot}). Now let $t>0$ and let $f$ be a monic real polynomial of degree $6t+2d$ (which is at least $16$). We apply Observation~\ref{obs4} and find a polynomial $h$ that divides $f$ and arises the characteristic polynomial of a matrix $M_1$ with sign pattern $T$. Using the inductive assumption, we find a matrix $M_2$ with characteristic polynomial $f/h$ and sign pattern that has the same form as $V$ but with one $T$-block removed. Now the matrix $\operatorname{diag}(M_1,M_2)$ has sign pattern $V$ and characteristic polynomial $f$.
\end{proof}

\begin{obs}\label{obs14}
For any family $\nu=(n_+,n_-,n_0,n_i)$ of nonnegative integers such that $n_++n_-+n_0+2n_i=8$, there is a family $\mu \leqslant \nu$ and a matrix $M$ with sign pattern $T$ and refined inertia $\mu$.
\end{obs}

\begin{proof}
If $n_0+2n_i\geqslant 6$, then we are done because of Observation~\ref{obs2}. Otherwise, we have $n_++n_-\geqslant3$, and it suffices to check that any tuple $\mu=(m_+,m_-,m_0,m_i)$  with $m_++m_-\geqslant3$ arises as a refined inertia of a matrix with sign pattern $T$.

Now we see that one of the tuples $\mu-(3,0,0,0)$, $\mu-(0,3,0,0)$, $\mu-(2,1,0,0)$, $\mu-(1,2,0,0)$ consists of nonnegative integers, and this tuple corresponds to some monic polynomial $h$ of degree $3$. We note that, for a sufficiently large positive $N$, the polynomials $(t-N)^3 h$, $(t+N)^3 h$,  $(t+3N) (t-N)^2 h$, $(t-3N) (t+N)^2 h$ satisfy the condition as in Observation~\ref{obs3}. As said above, one of these polynomials has $\mu$ as a refined inertia. 
%
%with refined inertia $\chi=(\chi_+,\chi_-,\chi_0,\chi_i)$. We note that, for a sufficiently large positive $N$, the polynomials $(t-N)^3 h$, $(t+N)^3 h$,  $(t+3N) (t-N)^2 h$, $(t-3N) (t+N)^2 h$ satisfy the condition as in Observation~\ref{obs3}. These polynomials have refined inertias $(3,0,0,0)+\chi$, $(0,3,0,0)+\chi$, $(2,1,0,0)+\chi$, $(1,2,0,0)+\chi$, and an appropriate choice of $h$ will allow us to get a polynomial that has a desired refined inertia 
\end{proof}

\begin{obs}\label{obs15}
The sign pattern $\operatorname{diag}(T,D)$ allows arbitrary refined inertias.
\end{obs}

\begin{proof}
Let $\nu=(n_+,n_-,n_0,n_i)$ be a family of nonnegative integers such that $n_++n_-+n_0+2n_i=8$. By Observation~\ref{obs14}, there is a family $\mu\leqslant\nu$ and a matrix $M_1$ with sign pattern $T$ and refined inertia $\mu$. Since $D$ is spectrally arbitrary, it allows a matrix $M_2$ with refined inertia $\nu-\mu$, and then the matrix $\operatorname{diag}(M_1,M_2)$ has sign pattern $\operatorname{diag}(T,D)$ and refined inertia $\nu$.
\end{proof}

Now we put all the observations together and conclude the paper.

\begin{thr}
Let $T,T',D$ be as above. Then

\noindent (1) the sign pattern $S=\operatorname{diag}(T,D,D,D,D,D)$ is spectrally arbitrary, but its superpattern $S'=\operatorname{diag}(T',D,D,D,D,D)$ is not spectrally arbitrary;

\noindent (2) $\operatorname{diag}(T,D)$ allows arbitrary refined inertias but is not spectrally arbitrary;

\noindent (3) there is a sign pattern $U$ such that $\operatorname{diag}(U,U)$ is spectrally arbitrary but $U$ is not.
\end{thr}

\begin{proof}
The definition of $T$ and $T'$ immediately shows that $S'$ is a superpattern of $S$. By Observation~\ref{obs1}, $S'$ does not allow a nilpotent matrix, so it is not spectrally arbitrary. Observation~\ref{obs5} shows that $S$ is spectrally arbitrary and completes the proof of~(1).

The sign pattern $\operatorname{diag}(T,D)$ is not spectrally arbitrary by Observation~\ref{obs12}, and it allows arbitrary refined inertias by Observation~\ref{obs15}. This proves~(2).

Finally, let $U_1=\operatorname{diag}(T,D)$. If $U_2=\operatorname{diag}(U_1,U_1)$ is spectrally arbitrary, then the proof of~(3) is complete. Otherwise, we define $U_3=\operatorname{diag}(U_2,U_2)$, and we are done if $U_3$ is spectrally arbitrary. If this is still not the case, we complete the proof because $\operatorname{diag}(U_3,U_3)$ is spectrally arbitrary by Observation~\ref{obs5}.
\end{proof}

\end{document}